\documentclass{elsarticle}[12pt]
\usepackage{amsmath, amsthm,amsfonts,amscd, amssymb, tikz, comment, mathrsfs, fullpage, cleveref}  
\usepackage{eucal} 	 
\usepackage{verbatim}      	
\usepackage{graphicx}
\makeatletter
\def\ps@pprintTitle{%
 \let\@oddhead\@empty
 \let\@evenhead\@empty
 \def\@oddfoot{}%
 \let\@evenfoot\@oddfoot}
\makeatother     	
\DeclareMathOperator{\Spec}{Spec}
\DeclareMathOperator{\height}{ht}

\DeclareMathOperator{\mlb}{Mlb}
\DeclareMathOperator{\mub}{mub}

\newtheorem{thm}{Theorem}[section]
\newtheorem{defn}[thm]{Definition}

\newtheorem{lem}[thm]{Lemma}
\newtheorem*{thmm*}{Theorem}

\newtheorem{notation}[thm]{Notation}
\DeclareMathOperator{\vp}{\varphi}

\begin{document}

\begin{abstract}
In this article, we develop a technique to ``split" certain types of partially ordered sets into simpler ones and use that technique to give a partial answer to a conjecture by R. Wiegand and S. Wiegand on the structure of semi-local, two-dimensional Noetherian spectra in their exposition on Noetherian prime ideals. Specifically, we show that very many two-dimensional posets with finitely many height two nodes that satisfy the necessary conditions of being the spectrum of a Noetherian ring, along with a very mild cardinality assumption, are in fact the spectrum of a Noetherian ring. 
\end{abstract}
\begin{keyword}
Noetherian spectra
\end{keyword}
\begin{frontmatter}
\author{Cory H. Colbert}
\title{On Two-Dimensional Semi-Local Noetherian Spectra}
\address{Williams College, Williamstown, MA. 01267}
\date{} 
\end{frontmatter}


\section{Introduction}
It is a wide-open question, originally posed by I. Kaplansky, as to which partially ordered sets (posets) arise as the spectrum of a commutative Noetherian ring. There have been many remarkable results and interesting insights in this area, and we refer the reader to a beautiful exposition written by Roger Wiegand and Sylvia Wiegand \cite{RSEXPO}. In 1983, S. Wiegand \cite{SPE} expanded and refined techniques developed by R. Heitmann in \cite{HEI1} and A. Doering and Y. Lequain in \cite{DL} to completely classify all countable, 2-dimensional, semi-local (finitely many maximal prime ideals) Noetherian spectra with a single minimal node. Left open were the cases where the spectra had several minimal prime ideals or were large in cardinality. It is these two issues that are of concern to us herein. 

En route to the discovery of our main theorem, stated below, we needed to develop a convenient notion for those posets $U$ that have a chance at being the spectrum of a semi-local, two-dimensional Noetherian ring. For example, it is trivial that $U$ must satisfy ACC on its nodes since the set of all ideals in a Noetherian ring satisfies ACC. As another example, any candidate poset $U$ must obviously only have finitely many minimal nodes since any commutative Noetherian ring has only finitely many minimal prime ideals. Other conditions exist on $U$ as well, and they are usually derived by invoking the prime avoidance lemma or using the altitude theorem. We discuss those conditions in much greater detail in section \ref{basics}, and we call any such candidate poset whose dimension is at most two a \textit{$K$-poset.} If $M$ is any height two node and $m$ is any height zero node in a $K$-poset $U$ such that $M > m,$ define $[M/m]$ to be all height one nodes $u$ in $U$ such that $u$ is exceeded only by $M$ and exceeds only $m.$ If the cardinality of $[M/m]$ is either zero or agrees with the cardinality of $U$ for all such pairs $M, m,$ then we call $U$ a \textit{proper $K$-poset.} Although this restriction looks a bit strong, it is actually rather mild in the sense that proper $K$-posets show up rather naturally; for example, if $R = k[x_1, \ldots, x_n]$ for $n \ge 2$ and $P, Q$ are two height two prime ideals of $R,$ then $\Spec R_{P \cup Q}$ is a proper $K$-poset. Our main theorem is thus stated:

\begin{thm}\label{MT} Every proper $K$-poset is the spectrum of some commutative Noetherian ring. \end{thm}

This theorem classifies a large collection of semi-local, two-dimensional Noetherian spectra, and it leaves only unusual $K$-posets not satisfying the mild cardinality restriction outside the scope of our approach. In fact, proper $K$-posets are only required to have finitely many height two nodes and we thus obtain spectra with infinitely many height one maximal prime ideals as part of our classification. It also offers a full resolution of the countable case of a conjecture on the structure of all semi-local, two-dimensional Noetherian spectra \cite[Conjecture 1.4]{RSEXPO}. 

Our approach is quite intuitive, and it is largely inspired by the techniques developed in \cite{DL}.  Given a proper $K$-poset $U,$ we seek to ``unravel" it into simpler pieces with simple inclusion relations. Outside of defining what it means to make a poset simpler, the main issue that presents itself here is to define what it means to have easy inclusion relations. In sections \ref{basics} and \ref{kclass}, we develop a notion of splitting a poset at a maximal node (much like splitting a prime ideal via an integral extension) and we describe and classify the proper $K$-posets that are, in a sense, the simplest one can hope to have (see Theorem \ref{classificationofkposets}). We call these atomic posets \textit{simple $K$-posets}, and we show that it is enough to realize those posets as spectra of Noetherian rings in order to prove Theorem \ref{MT}. 

In section \ref{embedding}, we show that if $U$ is a proper, simple $K$-poset, then we can find a Noetherian ring $R_0$ such that $U$ embeds into $\Spec R_0$ in the sense that they are almost identical, save for perhaps different sizes of certain equivalence classes of nodes. For example, perhaps $U$ has only two maximal nodes $m, n$, both height two, with only height one node $u$ dominated by them both, and maybe the corresponding maximal prime ideals $M, N$ in $R_0$ have infinitely many height one prime ideals contained in $M \cap N.$ In a similar fashion to how S. Wiegand proceeded in \cite{SPE}, we show, in section \ref{proof of main result}, that we can nevertheless find a Noetherian ring $S \supseteq R_0$ such that the equivalence classes in $\Spec S$ have the right size and everything else is otherwise preserved; that is, $U \cong \Spec S.$ 

\section{Acknowledgements}
The author would like to thank his doctoral supervisor, Raymond Heitmann, for his gracious support, enlightening advice, and clever insights during the preparation of the mathematical content of this article. The author is also grateful for the support, encouragement, and advice of Susan Loepp during the preparation of this article. Finally, the author would like to thank Williams College for their continued support.

\section{Basic Definitions and Notation}

Throughout this paper, all rings will be commutative with unity. If $f: A \to B$ is a map of rings, and $I$ is an ideal of $A,$ we will write $IB$ for the ideal $I^e,$ and if $J$ is an ideal of $B,$ we will write $J \cap A$ to denote the ideal $J^c.$

Let $U$ be a poset. We define $\dim U$ to be the supremum of the lengths of the chains of nodes in $U.$ If $V$ is another poset, we say $\vp: U \to V$ is a poset map iff $\vp(x) \le \vp(y)$ whenever $x \le y.$ We say $\vp$ is order-reflexive iff $x \le y$ whenever $\vp(x) \le \vp(y).$ Any order-reflexive map from $U$ onto $V$ is called an isomorphism of posets. If $A \subseteq U$ is a nonempty subset of nodes, we define $\mub A = \min \{u \in U: \text{$u \ge a$ for all $a \in A.$}\}.$ Similarly, we define $\mlb A = \max \{u \in U: \text{$u \le a$ for all $a \in A.$}\}.$ 

If $u \in U,$ we define $G(u) = \{v \in U: v \ge u\},$ and we define $L(u) = \{v \in U: v \le u\}.$ We define $G(u)^* = G(u)\setminus \{u\}$ and $L(u)^* = L(u) \setminus \{u\}.$  We define the height of a node $u \in U$ to be $\height u = \dim L(u),$ and we will let $H_i$ be the set of nodes of height equal to $i.$ 

Let $B, C \subseteq U.$ We define $[B/C]$ to be all nodes $u \in U$ such that $G(u)^* = B$ and $L(u)^* = C.$ We usually call this an intersection class of $U,$ and if either $B = \{b\}$ is a singleton (resp. $C = \{c\}$) we will usually write $[b/C]$ (resp. $[B/c]$). 

A path $(\gamma_t)_{t = 1}^n$ in $U$ is a finite sequence of nodes such that for all $1 < i \le n,$ the node $\gamma_i$ comparable to $\gamma_{i-1}.$ $U$ is said to be \textit{connected} iff for any two nodes $u, v \in U,$ there exists a path from $u$ to $v.$ 

In the case of ambiguity arising with any of this notation, we will affix subscripts or superscripts to the notation in appropriate places. For example, we may refer to height two nodes of $U$ as $H_2^U$ instead of just $H_2$ if the need arises.

\section{$K$-posets and Splittings}\label{basics}

\begin{defn} Let $U$ be a poset with $\dim U \le 2.$ We define $U$ to be a $K$-poset iff the following hold:
\begin{enumerate}
\item[1] $\min U, H_2$ are both finite;
\item[2] if $u, v$ are distinct nodes in $\min U,$ then $\mub\{u, v\}$ is finite; and
\item[3] whenever $u > v > w$ for some $v \in U,$ then $[u/w]$ is infinite.
\end{enumerate} 
Additionally, if for any maximal node $u$ and minimal node $w$ we have $|[u/w]| = 0$ or $|[u/w]| = |U|$ then we say $U$ is a \textit{proper} $K$-poset. \end{defn}

Note that we do not place a restriction on the number of height one nodes, and it is quite possible for the set of maximal nodes of $U$ to be infinite. We now provide numerous examples of interesting $K$-posets. 

\begin{thm}\label{SpecsAreKPosets}Let $R$ be a Noetherian ring with $\dim R = 2$ and for which there are only finitely many prime ideals of height two. Then $\Spec R$ is a $K$-poset. \end{thm}
\begin{proof} 

Items 1 and 2 are clear. 

Item 3: Assume that we have proven $\Spec R$ is a $K$-poset for all Noetherian rings $R$ for which $M \in \Spec R$ is maximal iff $\height M = 2.$ Let $R'$ be any Noetherian ring satisfying the conditions of the theorem. If $\dim R' \le 1,$ then there is nothing to show. If $\dim R' = 2,$ then localize at the union of all prime ideals of height two to form a new Noetherian ring $R.$ If $P \supsetneq P' \supsetneq P''$ for some primes $P, P', P''$ in $R',$ then $\height P = 2$ and thus survives the localization to $R.$ In particular, $PR \supsetneq P'R \subsetneq P''R$ and thus $[PR/P''R] \subset \Spec R$ is infinite. Each prime $QR \in [PR/P''R]$ is the extension of a prime $Q$ from $R'.$ Since $R'$ is Noetherian, all but finitely many such $Q$ must have $L(Q)^* = P''.$ Obviously a maximal prime ideal $M$ of $R$ has $M\supseteq Q$ iff $M = P.$  In particular, $[P/P''] \subset \Spec R'$ is infinite and $\Spec R'$ is thus seen to be a $K$-poset. We therefore need only prove the result for Noetherian rings $R$ for which every maximal prime ideal of $R$ has height two. 

Assume $R$ is indeed such a ring, and let $\mathcal H_{\Spec R}$ be the set of all minimal prime ideals of $R,$ along with all those height one prime ideals $P$ of $R$ for which $P$ dominates at least two different minimal primes of $R.$ Assume there is a chain $Q \supsetneq Q' \supsetneq Q''$ of prime ideals in $R$ for which $[Q/Q'']$ is finite. Let $M_1:=Q, M_2, \ldots, M_n$ be an enumeration of the maximal prime ideals of $R.$ Choose $x \in Q \setminus \bigcup_{i=2}^n M_i \bigcup_{P \in \mathcal H_{\Spec R} \cup [Q/Q'']} P$ by the prime-avoidance lemma. Define $I = (Q'', x).$ Since $I/Q''$ is principal and $R/Q''$ is Noetherian, the altitude theorem implies that there exists a prime ideal $Q^*/Q'' \supseteq I/Q''$ with $\height Q^*/Q'' = 1.$  If $\height Q^* = 2,$ then $Q^* = Q$ by choice of $x.$ But $Q\supsetneq Q' \supsetneq Q''$ implies that $\height Q/Q'' = 2.$ In particular, $\height Q^* = 1$ and of course $G(Q^*)^* = Q.$ Suppose $Q^* \supset W$ for some minimal prime ideal $W \ne Q''.$ Then $Q^* \supseteq Q'' + W$ and thus $Q \in \mathcal H_{\Spec R},$ conflicting with our choice of $x.$  Thus $L(Q^*)^* = Q''$ as well, and consequently $Q^* \in [Q/Q'']$ which is contrary to our assumptions.  \end{proof}

\begin{defn} Let $V$ be a poset containing a maximal node $m$ of positive height. We say a poset $U$ is a \textit{splitting} of $V$ at $m$ provided there exists a finite nonempty subset $\mathcal M \subseteq \max U$ of nodes of positive height and a surjective poset map $\varphi: U \rightarrow V$ satisfying the following conditions:

\begin{enumerate}
\item[1] $\varphi^{-1}(m) = \mathcal M,$ and $|\varphi^{-1}(v)| = 1$ for all $v \ne m.$ 
\item[2] If $\varphi(x') = x \le y$ for some $x' \in U, x, y \in V,$ then there exists $y' \in U$ for which $x' \le y'$ and $\varphi(y') = y.$ 
\end{enumerate}

If $U$ is a splitting of $V$ at $m$ with map $\varphi,$ then we call $\varphi$ a splitting map from $U$ to $V.$
\end{defn}

We wish to emphasize that we only split nodes $m$ of positive height, and if $U$ is a splitting of $V$ at $m,$ then $\mathcal M$ may not contain any height zero nodes.

\begin{lem}\label{dimension} Let $U$ be a splitting of $V$ at $m$ with splitting map $\varphi.$ Then $\dim U = \dim V$ and $\vp(\min U) = \min V.$ \end{lem}

\begin{proof} This is a straightforward consequence of the definitions. \end{proof}

\begin{notation} If $U$ is any $K$-poset we will define $\mathcal H_U$ to be the set of all minimal nodes of $U,$ along with those height one nodes that dominate at least two minimal nodes of $U.$ If $U$ is a poset with a single maximal node $m$ then we define $\mathcal H_U^* = \mathcal H_U \setminus\{m\}.$ \end{notation}

Unfortunately, splittings of $K$-posets need not be $K$-posets. The next lemma remedies this problem. 

\begin{lem}\label{refine} Let $U$ be a splitting of a proper $K$-poset $V$ at $m$ with splitting map $\varphi.$ Then there exists a proper $K$-poset $U' \supseteq U$ such that $\mathcal H_{U'} = \mathcal H_{U},$ $U'$ is a splitting of $V$ at $m$ with splitting map $\varphi' : U' \rightarrow V$ such that $\varphi' = \varphi$ on $\mathcal H_{U'}.$  \end{lem}

\begin{proof} 
If $\dim V < 2,$ then so is $\dim U$ by Lemma \ref{dimension}, and therefore $U$ is a proper $K$-poset in that case. Now assume $\dim U = \dim V = 2.$ Observe that the properties of splitting maps force $|U| = |V|,$ $|\min U| = |\min V| < \infty,$ and $|H_2^U| < \infty.$  Moreover, if $u, v \in \min U,$ then $\mub\{u, v\}$ is finite since otherwise $\mub\{\varphi(u), \varphi(v)\} \cap H_1 \subseteq V$ would be infinite which contradicts the assumption that $V$ is a $K$-poset. Properties 1 and 2 of the definition of a proper $K$-poset are thus verified for $U.$ At this stage, one need only enlarge deficient classes of the form $[a/c]$ to turn $U$ into a proper $K$-poset. The details are mundane but straightforward.  \end{proof}

If $V$ is a poset with a single maximal node $m,$ we will define a quantity $d_V(m)$ to measure how complicated $\mathcal H_V$ is. 

\begin{notation} Assume $\dim V \ge 1.$ Define $\Lambda_V \subseteq \mathcal H_V^*$ to be $H_1 \cap \mathcal H_V^*$ along with all those nodes $v \in \mathcal H_V$ for which $G(v) \cap H_1 \cap \mathcal H_V^*$ is empty. Define $d_V(m) = |\Lambda_V|.$ \end{notation} 

Note that if $\dim V > 0$ then $d_V(m) \ge 1$ since if $H_1 \cap \mathcal H_V^*$ is empty then everything in $\min V$ vacuously lives in $\Lambda_V.$ Also, in Lemma \ref{refine} we have $\mathcal H_U = \mathcal H_{U'}$ so that $d_{L_U(m_i)}(m_i) = d_{L_{U'}(m_i)}(m_i)$ for all $m_i \in \mathcal M.$

The next theorem shows that if $V$ is a proper $K$-poset with a single maximal node $m$ such that $d_V(m) > 1,$ then we can find a new proper $K$-poset $U$ that splits $V$ and has much simpler inclusions with respect to $d.$

\begin{thm}\label{localsimplify} Let $V$ be a proper $K$-poset with a single maximal node $m.$ If $d_V(m) > 1$ then there exists a proper $K$-poset $U$ which is a splitting of $V$ at $m$ such that $d_{L_U(m_i)}(m_i) = 1$ for all $m_i \in \mathcal M \subseteq \max U.$ \end{thm}

\begin{proof} 
First we make an observation about $\Lambda_V:$ If $a, b$ are distinct elements in $\Lambda_V,$ then $G_V(a) \cap G_V(b) = \{m\}.$ Indeed, if $t \ge a$ and $t \ge b$ then $\height t \ge 1.$ If $\height t = 2,$ then $t = m.$ Assume $\height t = 1$ and $t \ne m.$ If both $a$ and $b$ are minimal, then $t \in H_1\cap \mathcal H_V^*,$ and this conflicts with $G_V(a) \cap H_1 \cap \mathcal H_V^*$ being empty. Thus, $\height a = 1$ without loss of generality and $t = a \ge b.$ Seeing as $a \in \Lambda_V,$ we must have $a \in H_1 \cap \mathcal H_V^*.$ But this is impossible since then $G_V(b) \cap H_1 \cap \mathcal H_V^*$ would be nonempty.  Therefore, $t = m.$ 

Enumerate the elements of $\Lambda_V$ as $a_1, \ldots, a_n.$ Let $\mathcal M$ be any set of $n$ elements disjoint from $V$ enumerated as $m_1, \ldots, m_n.$ Define $W:= V\setminus\{m\} \cup \mathcal M.$ We place an order relation on $W.$ First declare $u \le_{W} u$ for all $u \in W,$ and let $V\setminus\{m\}$ inherit its order from $V.$ If $m_i \in \mathcal M$ and $x \in W,$ set $m_i \le_{W} x$ iff $x = m_i.$ We thus only need to relate the set $\mathcal M$ to $V\setminus\{m\}$ in a sound manner. Suppose $v \in V\setminus\{m\}$ and $m_i \in \mathcal M.$ Note that $\height_W v \le 1.$ We will use the set $L_V(v) \cap \Lambda_V$ to determine how to relate $v$ to $m_i.$ Either it is empty or not. If it is nonempty, then the previous paragraph implies that $L_V(v) \cap \Lambda_V = \{a_j\}$ for some $a_j \in \Lambda_V.$ In this case, we declare $v \le_{W} m_i$ iff $i = j.$ If it is empty, we claim that $v$ dominates exactly one minimal node $v'.$ Indeed, if $v \in \mathcal H_V^*$ then $v \in \min V$ since otherwise $v$ would live in $\Lambda_V.$ In such case, we have $v = v'.$ If $v \notin \mathcal H_V^*,$ then by definition $v$ must dominate a single minimal node $v' < v.$ Since $L_V(v) \cap \Lambda_V$ is empty, it follows that $v' \notin \Lambda_V$ and the set $G(v') \cap H_1 \cap \mathcal H_V^*$ is nonempty. Declare $v \le_{W} m_i$ iff  $a_i \in G(v') \cap H_1 \cap \mathcal H_V^*.$ 

We claim these relations turn $W$ into a poset. We need only check transitivity since the other relations are clear. Assume $a < b < c,$ and not all three belong to $V\setminus\{m\}.$ If $c = m_i$ for some $i,$ then $a, b \in V\setminus \{m\}$ and satisfy $a <_V b.$ Note that $a \in \min V$ since $\dim V = 2.$ We must show $a <_{W} m_i.$ If $L_V(b) \cap \Lambda_V$ is empty, then $a$ must be the unique node dominated by $b.$  Since $b<_{W} m_i,$ we must have $a_i \in G(a) \cap H_1 \cap \mathcal H_V^*,$ and therefore $a<_{W} m_i$ by definition. If $L_V(b) \cap \Lambda_V$ is nonempty, then $b <_{W} m_i$ means that $L_V(b) \cap \Lambda_V = \{a_i\}.$ Either $a = a_i$ or $b = a_i$ (since $\height_W b = 1$). In either case, we have $a <_W m_i.$ 

We claim $\Lambda_{L_{W}(m_i)} = \{a_i\}.$ First, we show $a_i \in \Lambda_{L_{W}(m_i)}.$ Since $a_i <_{W} m_i$ by definition of $\le_{W},$ we see that $a_i \in L_{W}(m_i).$ If $\height a_i = 1$ in $V,$ then $a_i \in \Lambda_V$ implies that $a_i \in H_1 \cap \mathcal H_V^*.$ Since $V \setminus\{m\} \subset W$ inherits its order from $V,$ it follows that $a_i \in H_1 \cap \mathcal H_{L_{W}(m_i)}^* \subset \Lambda_{L_{W}(m_i)}.$ If $\height a_i = 0$ in $V,$ then $a_i \in \min V = \min W.$ Assume there is $b \in G_{L_{W}(m_i)}(a_i) \cap H_1 \cap \mathcal H_{L_{W}(m_i)}^*.$ Since $b$ dominates at least two minimal nodes in $W,$ the same must be true in $V,$ and we see that in fact $b \in H_1 \cap \mathcal H_V^*.$ Of course, $b \in G_V(a_i)$ as well, and it follows that $G_V(a_i) \cap H_1 \cap \mathcal H_V^*$ is nonempty, contradicting $a_i \in \Lambda_V.$ Therefore $a_i \in \Lambda_{L_{W}(m_i)}.$ Now we show if $x \in \Lambda_{L_{W}(m_i)},$ then $x = a_i.$ Indeed, if $x \in H_1,$ then $x \in \Lambda_V.$ In this case, $x = a_j$ for some $j,$ and since $x <_W m_i$ we must have $j = i$ so that $x = a_i.$ If $x \notin H_1,$ then it must be minimal in $W.$ If $L_V(x) \cap \Lambda_V$ is nonempty, then of course $x = a_i.$ Otherwise, we must have $a_i \in G_V(x) \cap H_1 \cap \mathcal H_V^*$ which implies that $a_i \in G_{L_W(m_i)}(x) \cap H_1 \cap \mathcal H_{L_W(m_i)}^*,$ contradicting $x \in \Lambda_{L_W(m_i)}.$ In particular, $d_{L_W(m_i)} = 1.$ 

Let $\varphi': {W} \rightarrow V$ be the map $\varphi'(u) = u$ if $u \in V\setminus \{m\},$ and $\varphi'(u) = m$ if $u \in \mathcal M.$ Since each $u \in W$ is dominated by some $m_i \in \mathcal M,$ the second property of splitting maps is surely satisfied and it is easy to see that $\varphi'$ is a splitting map turning $W$ into a splitting of $V$ at $m.$ By Lemma \ref{refine}, there exists a proper $K$-poset $U \supseteq W$ with $|U| = |V|, \mathcal H_U = \mathcal H_{W},  \varphi' = \varphi$ on $\mathcal H_{W},$ and for which $U$ is a splitting of $V$ at $m$ with splitting map $\varphi.$ Since $\mathcal H_U = \mathcal H_{W},$ it follows that $\mathcal H_{L_U(m_i)} = \mathcal H_{L_{W}(m_i)}$ so that $d_{L_U(m_i)}(m_i) = d_{L_{W}(m_i)}(m_i) = 1$ for all $m_i \in \mathcal M.$ \end{proof}

\section{Simple $K$-posets and Classifications}\label{kclass}
\begin{defn} A proper $K$-poset $V$ is called \text{simple} iff $d_{L_V(m)}(m) = 1$ for all nodes $m \in \max V$ of positive height. \end{defn} 

 Let $\alpha$ be a cardinal number. We define three kinds of posets. If $V = \{x\}$ with the trivial order relation $x \le_V x,$ then $V$ is said to be a point. If $\dim V = 1, |\min V| = 1, |\max V| = \alpha$ then $V$ is said to be an $\alpha$-fan. Now let $k$ be a positive integer, let $\alpha_1, \ldots, \alpha_k$ be $k$ disjoint sets of cardinality $\alpha,$ and let $m, t, t_1, \ldots, t_k$ be $k+2$ separate elements. Let $V = \{m, t, t_1, \ldots, t_k\}\cup_{i=1}^k \alpha_i.$ We put an order relation on $V.$ The nodes $t_1, \ldots, t_k$ will be minimal nodes of $V.$ Declare $t > t_i$ for all $1 \le i \le k.$ Declare $m \ge x$ for all $x \in V.$ Finally, demand that $y \ge t_i$ iff $y = t_i, y= t, y\in \alpha_i$ or $y = m.$ If $x, y \in \{t\} \cup \cup_{i=1}^k \alpha_i,$ then declare $x \le y$ iff $x = y.$ We call $V$ an $\alpha$-tent. The following familiar poset $T$ is an $\aleph_0$-tent: 
 
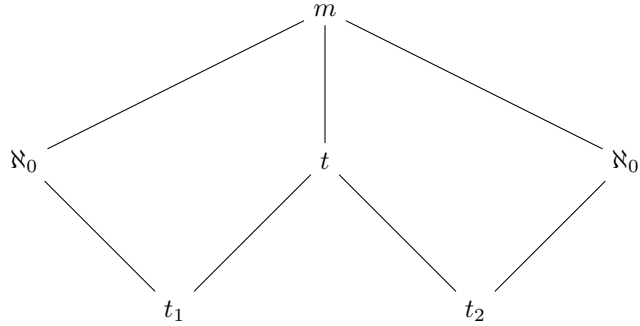
\begin{center}
\begin{tikzpicture}[scale=2]
\node(p1) at (-1,0) {$t_1$};
\node(p2) at (1,0) {$t_2$};
\node(t) at (0, 1) {$t$};
\node(inf1) at (-2,1) {$\aleph_0$};
\node(inf2) at (2, 1) {$\aleph_0$};
\node(m) at (0, 2) {$m$};
\draw (t) -- (p1) -- (inf1) -- (m) -- (t) -- (p2) -- (inf2) -- (m);
\end{tikzpicture}
\end{center}

 \begin{thm}\label{classificationofkposets} If $V$ is an $\alpha$-tent, an $\alpha$-fan, or a point, then $V$ is simple. Conversely, if $V$ is a simple, proper $K$-poset with a single maximal node $m,$ then $V$ is either a point, a 1-fan, or a $\beta$-tent for some cardinal number $\beta.$ \end{thm}
 \begin{proof}
If $V$ is either a point or a fan, then $V$ is simple. In the former case, there is nothing to prove, and in the latter case clearly $\Lambda_{L_V(m)} = \min V$ for all $m \in \max V.$ Assume $V$ is an $\alpha$-tent. Suppose that $\min V = \{t_1\}.$ Since $V$ has only one minimal node, it follows that $H_1 \cap \mathcal H_V^*$ is empty. In particular, $G(t_1) \cap H_1 \cap \mathcal H_V^*$ is empty so that $t_1 \in \Lambda_V.$ It is clear that no other node in $V$ can belong to $\Lambda_V,$ and therefore $d_V(m) = 1$ in the case where $|\min V| = 1.$ Now assume $|\min V| > 1.$ Note that $\mathcal H_V^* = \{t, t_1, \ldots, t_k\},$ where $k > 1.$ Since $t > t_i$ for all $1 \le i \le k,$ we must have $\Lambda_V = \{t\}$ and $d_V(m) = 1.$ 

Now assume $V$ is a simple poset with a single maximal node $m.$ If $\dim V = 0,$ then $V$ is clearly a point. If $\dim V = 1,$ then $d_V(m) = 1$ means that $|\Lambda_V| = 1.$ Since $\mathcal H_V^* = \min V,$ we must have $\min V \subseteq \Lambda_V$ since trivially $G(a) \cap H_1 \cap \mathcal H_V^*$ is empty for all $a \in \min V.$ In particular, $|\min V| = 1$ and $V$ must be a 1-fan. Assume $\dim V = 2.$ Write $\beta = |V|, k = |\min V|, \Lambda_V = \{q\}.$ If $k = 1,$ then $V$ is clearly a $\beta$-tent since it is a proper $K$-poset. In this case, $q$ is the unique minimal node of $V.$ Assume $k > 1.$ We first show that $q$ cannot be minimal. Assume the contrary, and write $q = t_j$ for some $1 \le j \le k.$ Since $\height q = 0,$ we must have that $H_1 \cap \mathcal H_V^*$ is empty. Since $k > 1,$ there exists $t_{j'}$ different from $q,$ and both $G(t_{j'})$ and $G(q)$ must vacuously miss $H_1 \cap \mathcal H_V^*.$ However, this forces $|\Lambda_V| \ge 2,$ which contradicts simplicity. In particular, $\height q = 1.$ Since $V$ is simple, no other height one node belongs to $\mathcal H_V^*,$ and it follows that any node $q'$ dominating at least two minimal nodes of $V$ must coincide with $q.$ In particular, if $t_l \not < q$ for some $l,$ then $G(t_l)$ must miss $H_1 \cap \mathcal H_V^*$ which implies $|\Lambda_V| \ge 2.$ Therefore, $t_l \le q$ for all $1 \le l \le k,$ and since $V$ is a proper $K$-poset, we see that $[m/t_l]$ has $\beta$ nodes for all $1 \le l \le k.$ $V$ is thus seen to be a $\beta$-tent. \end{proof} 
Notice that the minimal upper bound information in tents, fans and points is quite easy to understand. This is attractive to us because if we can reduce the spectra we need to construct to such ``simple" posets, then we would not ever be required to find Noetherian spectra with hard minimal upper bound sets. 

\begin{defn} Let $$\mathcal V:  \ldots \rightarrow V_{\ell} \rightarrow V_{\ell-1} \rightarrow \ldots \rightarrow V_1 \rightarrow V_0 := V$$ be a chain of proper $K$-posets for which $V_i$ is a splitting of $V_{i-1}$ at some maximal node $n_i \in V_{i-1}$ for all $i \ge 1,$ and $\varphi_i: V_i \rightarrow V_{i-1}$ is a splitting map. We call $\mathcal V$ a \textit{simplifying chain} of $V.$ If there exists $\ell \ge 0$ for which $V_{\ell}$ is simple, we refer to $V_{\ell}$ as a \textit{simplification} of $V.$ \end{defn}

We aim to show that every proper $K$-poset has a simplification. Note that Theorem \ref{localsimplify} accomplishes this in the case where $V$ has a single maximal node $m.$ In order to accomplish this in general, we will need a lemma that allows us to extend splittings of subposets of $V$ to the whole of $V.$ 

\begin{lem}\label{expand} Let $X, Z$ be posets and let $f$ be a poset map from $X$ to $Z$ such that whenever $x \in X,$ we have $L_Z(f(x)) \subseteq f(X).$ Then there exists a poset $Y \supseteq X$ and a surjective poset map $g: Y \rightarrow Z$ for which the following statements hold:
\begin{enumerate}
\item[1] We have $g|_X = f.$ 
\item[2] If $f(n) \in \max Z,$ then $n \in \max Y$ and $d_{L_X(n)}(n) = d_{L_Y(n)}(n).$ 
\item[3] If $\max f(X) \subseteq \max Z,$ and $f$ is a splitting map from $X$ to $f(X),$ then $g$ is a splitting map from $Y$ to $Z.$ 
\end{enumerate}
\end{lem}

\begin{proof} 

\textbf{Item 1:} Let $Y' := Z\setminus f(X) \subseteq Z$ with the induced order from $Z$ placed on $Y'.$   Let $Y:= X \cup Y'$ be the disjoint union of $Y'$ and $X.$ Note that we have a natural surjective set map $g: Y \rightarrow Z$ which is given by $g(a) = f(a)$ if $a \in X$ and $g(a) = a$ if $a \notin X.$ 

We place an order $\le_Y$ on $Y.$ Let $a, b \in Y.$ If $a, b \in X$ then put $a \le_Y b$ iff $a \le_X b$ in $X.$ If $a, b \in Y',$ then put $a \le_Y b$ iff $a \le_Z b$ in $Z.$ Otherwise, put $a \le_Y b$ iff $g(a) \le_Z g(b)$ in $Z.$ It is clear that $a \le_Y a$ for all $a \in Y.$ Suppose $a \le_Y b$ and $b \le_Y a$ for $a, b \in Y.$ If both $a, b$ lie in either $X$ or $Y',$ then of course $a = b.$ It is not possible, in this case, for $a$ to live in one of $X$ or $Y'$ and $b$ to live in the other. Indeed assume, without loss of generality, that $a \in X$ and $b \in Y'.$ Then $f(a) \le_Z b$ and $b \le_Z f(a)$ so that $b = f(a).$ In particular, $b \in f(X)$ and $b \notin f(X),$ a contradiction. It follows that $\le_Y$ is reflexive and anti-symmetric. 

We now prove that $\le_Y$ is transitive. First, we make a brief observation. If $u \le_Y v,$ and $v \in X$ then $u \in X$ as well. If $u \notin X,$ then $u \in Z\setminus f(X)$ and $u = g(u) \le_Y f(v) \in f(X).$ By our assumption on $f,$ we have $u \in L_Z(f(v)) \subseteq f(X),$ which is a contradiction. Therefore, $u \in X.$ Consequently, if $a \le_Y b \le_Y c,$ and either $c \in X$ or $a \notin X,$ then transitivity is immediate. Therefore, assume $a \in X$ and $c \notin X.$ If $b \in X,$ then we must obtain $f(b) \le_Z c$ by definition of the order relation on $Y,$ and since $f: X \to Z$ is a poset map, we must also have $f(a) \le_Z f(b)$ so that $f(a) \le_Z c.$ This implies $a \le_Y c$ in $Y.$ If $b \notin X,$ then $b \le_Z c$ and $f(a) \le_Z b$ by definition of the order relation on $Y.$ Since $Z$ is a poset, we have $f(a) \le_Z c$ in $Z,$ and therefore $a \le_Y c$ again by the order relations. 

\textbf{Item 2:} Suppose $n \le n'$ for some $n' \in Y.$ If $n' \notin X,$ then $f(n) <_Z n',$ which is nonsense since $f(n) \in \max Z.$ The second part of item 2 follows from the first part of the previous paragraph which asserts, in our case, that $L_X(n) = L_Y(n).$

\textbf{Item 3:} Suppose $f$ is a splitting map from $X$ to $f(X)$ for some $m \in \max f(X) \subseteq \max Z.$ Let $\mathcal M \subseteq X$ be the set of maximal nodes of $X$ that agree under $f.$ It is clear that $\mathcal M \subseteq \max Y$ and $g^{-1}(m) = \mathcal M.$ Moreover, if $z \in Z,$ then it is the image of exactly one node from $Y,$ unless $z = m,$ since $f$ is a splitting map, and $g$ is the identity outside of $X.$ Now we show that $g$ has the required lifting property of splitting maps. Assume $g(a_1) = b_1 \le_Z b_2.$ If $b_1, b_2$ both belong to either $Y'$ or $f(X)$ then clearly there exists $a_2 \ge_Y a_1$ for which $g(a_2) = b_2.$ Otherwise, if $b_1 \in Y'$ and $b_2$ is not, then upon choosing $a_2$ for which $g(a_2) = b_2$ we see that $a_1 \le_Y a_2$ automatically by definition of the order on $Y.$ \end{proof}

\begin{lem}\label{preserveD} If $U$ is a splitting of $V$ at $m$ with splitting map $\varphi,$ then $d_{L_U(n)}(n) = d_{L_V(\varphi(n))}(\varphi(n))$ for all $n \notin \mathcal M.$ \end{lem}
\begin{proof}
This is a straightforward consequence of the splitting properties of $\varphi.$ \end{proof}

\begin{thm}\label{globalsimplify} Every proper $K$-poset $V$ has a simplification. \end{thm}

\begin{proof} Let $V:=V_0$ be a proper $K$-poset, and assume there is a maximal node $m \in \max V_0$ for which $d_{L_{V_0}(m)}(m) > 1.$ Choose, by Theorem \ref{localsimplify}, a proper, simple $K$-poset $U_1$ that is a splitting of $L_{V_0}(m)$ at $m,$ and that satisfies $d_{L_{U_1}(m_i)}(m_i) = 1$ for all $m_i \in \mathcal M.$ Of course, $U_1$ may not be a splitting of $V_0$ at $m,$ and we must now invoke Lemma \ref{expand} to find poset $W_1 \supseteq U_1$ such that both $W_1$ is a splitting of $V_0$ at $m$ and $d_{L_{U_1}(m_i)}(m_i) = d_{L_{W_1}(m_i)}(m_i) = 1$ for all $m_i \in \mathcal M = \max W_1 = \max U_1.$ Now $W_1$ may not be a proper $K$-poset, so we will simply apply Lemma \ref{refine} to obtain a proper $K$-poset $V_1 \supseteq W_1$ for which $d_{L_{W_1}(m_i)}(m_i) = d_{L_{V_1}(m_i)}(m_i)$ for all $m_i \in \mathcal M\subseteq \max V_1,$ and for which $V_1$ is a splitting of $V_0$ at $m.$ By Lemma \ref{preserveD}, we have $d_{L_{V_1}(n)}(n) = d_{L_{V_0(\varphi(n))}}(\varphi(n))$ for all $n \notin \mathcal M.$ Therefore, if $n \in \max V_1$ either $d_{L_{V_1}(n)}(n) = 1$ or agrees with $d_{L_{V_0}(\varphi(n))}(\varphi(n)).$ In particular, if $e(U)$ is the number of maximal nodes $n$ of a proper $K$-poset for which $d_{L_U(n)}(n) > 1,$ then $e(V_1) < e(V).$ An induction on $e(V)$ will therefore obtain the desired result. \end{proof}

\section{Gluing maximal ideals and splitting $K$-posets}\label{duality}

\begin{thm}\label{iso} Let $U$ be a splitting of $V$ at $m$ with splitting map $\varphi,$ and suppose $U \cong \Spec R$ for some Noetherian ring $R$ for which $R/M \cong R/N$ for all maximal prime ideals $M, N$ of $R.$ Then there exists a Noetherian ring $L \subseteq R$ for which $\Spec L \cong V.$ \end{thm} 

\begin{proof} Let $\rho: \Spec R \to U$ be an isomorphism of posets. Then $\varphi \rho: \Spec R \to V$ is a splitting of $V$ at $m.$ Define $\mathcal M = \rho^{-1}\varphi^{-1}(m),$ and glue the maximal prime ideals of $\mathcal M$ together via \cite[Theorem A]{DL} to form a Noetherian subring $L$ of $R.$ Let $M \in \Spec L$ be the conductor $\bigcap_{N \in \mathcal M} N$ of $R$ into $L.$

We now construct an isomorphism from $\Spec L$ onto $V.$ For every $P \in \Spec L,$ let $\pi(P)$ be any choice of lifting into $\Spec R.$ Define $\tau: \Spec L \to V$ to be $\tau(P) = \varphi\rho\pi(P).$ Note that if $\pi(P) \ne \pi(Q),$ then $\pi(P), \pi(Q) \in \mathcal M,$ and therefore $\varphi(\rho(\pi(P))) = \varphi(\rho(\pi(Q))) = m.$  That is, $\tau$ is well-defined as a map of sets. 

Now we show $\tau$ is a poset map with all the required properties. If $P \subseteq Q$ in $\Spec L,$ and $Q \ne M,$ then ``Theorem A" implies that there is only one choice of lifting of $P$ and $Q$ into $\Spec R,$ and thus $\tau(P) \le \tau(Q)$ is immediate. If $Q = M,$ then $\pi(P) \subseteq N$ for some $N \in \mathcal M,$ and consequently $$\varphi(\rho(\pi(P))) = \tau(P) \le m = \tau(Q).$$ Now assume $\tau(P) \le \tau(Q)$ in $V.$ If $\tau(Q) \ne m,$ then the lifting property of splitting maps ensures that $\rho(\pi(P)) \le \rho(\pi(Q))$ and therefore $\pi(P) \subseteq \pi(Q)$ since $\rho$ is an isomorphism. This of course implies that $P \subseteq Q.$ If $\tau(Q) = m,$ then $Q$ must coincide with $M,$ and we must have $\pi(P) \subseteq N'$ for some $N' \in \mathcal M.$ In this case, $P \subseteq N' \cap L = M = Q$ still. Therefore, $\tau$ is an order-reflective poset map. That $\tau$ is onto is clear.\end{proof}

\begin{lem}\label{primeavoidance} Let $R$ be a Noetherian ring, and suppose there exists a set $B$ of $\beta$ elements of $R$ for which $u - v$ is a unit of $R$ for all $u \ne v \in B.$ Let  $\{P_{\lambda}\}_{\lambda \in \Lambda}$ be a set of $\alpha$ prime ideals of $R,$ and assume that $I \subseteq \bigcup_{\lambda \in \Lambda} P_{\lambda} $ for some ideal $I$ of $R.$ If $\alpha < \beta,$ then $I \subseteq P_{\gamma}$ for some $\gamma \in \Lambda.$ \end{lem} 

\begin{proof} The proof of this lemma is almost identical to the proof of \cite[Lemma 14.2]{CMR}. \end{proof} 

Consequently, if a Noetherian $R$ contains a field $k$ of cardinality $\beta,$ then we may set $B = k$ and apply Lemma \ref{primeavoidance} to avoid sets of $\alpha$ prime ideals for $\alpha < \beta.$ 

We wish to significantly reduce the problem of classifying proper $K$-posets to that of classifying proper, simple $K$-posets with finitely many maximal nodes. The reduction to the case where there only finitely many maximal nodes is a fairly straightforward matter. Indeed, if $U$ is any proper $K$-poset, then there exists a proper $K$-poset $U' \supseteq U$ with finitely many maximal nodes and satisfying $\mathcal H_U = \mathcal H_{U'}.$ Realizing $U'$ as $\Spec R'$ for some Noetherian ring $R',$ we see that a suitable localization $S^{-1}R'$ of $R'$ will give $\Spec S^{-1}R' \cong U.$ The next result gives us the rest of our desired reduction.

\begin{thm}\label{reduce} If all proper, simple $K$-posets of cardinality $\beta$ and $|\max U| < \infty$ are realizable as the spectrum of a commutative Noetherian ring with common residue field $k$ for some field $k,$ then all proper $K$-posets of cardinality $\beta$ with $|\max U| < \infty$ are realizable as the spectrum of a Noetherian ring with common residue field $k.$ \end{thm}

\begin{proof}
Let $V$ be a proper $K$-poset with finitely many maximal nodes. Then $V$ has a simplification $V_{\ell}$ into a proper, simple $K$-poset with finitely many maximal nodes by Theorem \ref{globalsimplify}. Realize $V_{\ell}$ as the spectrum of a suitable commutative Noetherian ring $R$ with common residue field $k.$ Then, by Theorem \ref{iso} we see that $V \cong \Spec S$ for some subring $S \subseteq R.$ By property ``c" of  \cite[Theorem A]{DL}, all the residue fields of $S$ are still $k$ and the result is proved. 
\end{proof} 

\section{Embedding $K$-posets into Noetherian spectra}\label{embedding}
We begin working to show that if $U$ is any proper, simple, connected $K$-poset then there is a Noetherian ring for which $U$ embeds into $\Spec R$ in a strong sense: the map $\varphi$ will induce an isomorphism from $\mathcal H_U$ onto $\mathcal H_{\Spec R},$ and send $\max U$ onto $\max \Spec R.$ Moreover, it will send classes of the form $[\mathcal M/u]$ into classes of the form $[\varphi(\mathcal M)/\varphi(u)].$

\begin{lem}\label{genericufd} Let $R$ be a Noetherian domain containing an infinite field $k$ of cardinality $\beta,$ and let $X$ be a set of indeterminates of cardinality $\beta.$ Let $M_1, \ldots, M_n$ be maximal prime ideals of $R$ of height two. If $1 \le l \le n$ is any integer, and $X' \subset X$ is any subset of indeterminates, and $x_t \notin X'$ is another indeterminate, then there are $\beta$ finitely generated height one prime ideals $Q$ of $R[X'][x_t]$ satisfying $Q \subset \bigcap_{i=1}^l M_iR[X'][x_t], Q\not \subseteq \bigcup_{j = l+1}^n M_jR[X'][x_t],$ and $Q \cap R[X'] = 0.$ \end{lem} 

\begin{proof} The proof of this lemma is similar to the proof of \cite[Lemma 2]{MCADAM}. The first observation we make is that if $a \in R[X'],$ then there are only finitely many height one prime ideals in $R[X']$ containing $a.$ To see this, choose a finite subset $F \subseteq X'$ for which $a \in R[F].$ Then $R[F]$ is Noetherian and of course there are only finitely many height one prime ideals $P_1, \ldots, P_w$ containing $a.$ If $P$ is any height one prime ideal of $R[X']$ containing $a,$ then $P \cap R[F]$ contains $a.$ If $\height P \cap R[F] > 1,$ then there is a finitely generated prime ideal $P'$ whose height is greater than one and satisfies $P'R[X'] \subset P.$ By \cite[Lemma 7]{JSPEC}, we have $\height P'R[X'] > 1,$ which conflicts with $\height P = 1.$ In particular, $\height P \cap R[F] = 1$ and $P$ is easily seen to be extended from $R[F].$ 

Choose nonzero $f$ in $M_1R[X'] \cap \ldots \cap M_nR[X'],$ and choose $g \in M_1R[X'] \cap \ldots \cap M_lR[X'] \setminus \cup_{j=l+1}^n M_jR[X']$ and outside the minimal prime ideals of $f$ in $R[X']$ of which there are only finitely many by the above argument. For each $a \in k^{\times},$ set $z_a = fx_t + ag.$ Note that $z_a$ is not in $M_jR[X'][x_t]$ for any $l < j \le n.$ 

We claim that there is a height one prime ideal $Q_a$ satisfying $Q_a \subset \bigcap_{i=1}^l M_iR[X'][x_t]$ and $Q_a \not \subseteq \bigcup_{j=l+1}^n M_jR[X'][x_t]$ containing $z_a$ for each $a \in k^{\times}.$  Localizing at the union of $M_iR[X', x_t],$ we have formed by \cite[Lemma 7]{JSPEC} a Noetherian ring of Krull dimension 2. In particular, $z_a$ lies in a height one prime ideal $Q_a'$ extended from a height one prime ideal $Q_a$ in $R[X', x_t]$ contained in the union of the prime ideals $\bigcup_{i=1}^n M_iR[X', x_t].$ Clearly $Q_a$ is outside of $\bigcup_{j=l+1}^n M_j$ since $z_a$ is not in that union. We thus need to show that $Q_a$ is the only height one prime ideal containing $z_a.$ We show $Q_a \cap R[X'] = 0.$ Assume otherwise. If $Q_a \cap R[X'] \ne 0,$ then it has height one. To see this, localize at any of the $M_iR[X']$ containing $Q_a \cap R[X']$ for $1 \le i \le l$ and apply the altitude theorem. In particular, $Q_a = Q^*R[X'][x_t]$ for some prime ideal $Q^* \subset R[X'].$ Since $fx_t + ag \in Q_a,$ this means that $f, g \in Q^*$ by definition of ideal extension, which contradicts our choice of $f, g$ above. In particular $Q_a$ contracts to zero and survives the localization to $K[x_t]$ where $K$ is the fraction field of $R[X'].$ Just as in \cite[Lemma 2]{MCADAM}, the only prime ideal so surviving is clearly $(x_t - ag/f)K[x_t] \cap R[X'][x_t].$ 

To close the proof, we just need to check that we have generated $\beta$ prime ideals in $\bigcap_{i=1}^l M_iR[X'][x_t]$ but not $\bigcup_{j=l+1}^n M_jR[X'][x_t],$ all of which contract to zero. If $z_a, z_b \in Q_a$ for some distinct $a, b \in k^{\times},$ then $(a - b)g \in Q_a$ forcing $g \in Q_a.$ Then $g \in Q_a \cap R[X'] = 0,$ which is an absurdity. \end{proof}

\begin{thm}\label{Embed} Let $U$ be a proper, simple, $K$-poset of cardinality $\beta$ such that $|\max U| < \infty.$ Let $k = k'(T)$ where $T$ is a set of $\beta$ indeterminates over a countable field $k'$ of characteristic zero. There exists a commutative Noetherian ring $R$ and poset embedding $\vp: U \rightarrow \Spec R$ for which the following items hold:
\begin{enumerate}
\item[1] $R$ is a $k$-algebra with $|R| = |k| = \beta,$ and all the residue fields of $R$ are $k;$ 
\item[2] $\varphi$ is height-preserving and induces an isomorphism from $\mathcal H_U$ onto $\mathcal H_{\Spec R};$
\item[3] $\varphi[\mathcal N/u] \subseteq [\varphi(\mathcal N)/\varphi(u)]$ for all classes $[\mathcal N/u];$ and 
\item[4] $\varphi(\max U) = \max \Spec R.$
\end{enumerate}
\end{thm} 
\begin{proof} It suffices to assume $U$ is connected.  Enumerate the minimal nodes of $U$ as $u_1, \ldots, u_n.$ 

Assume $n =1.$ Enumerate the maximal nodes of $U$ as $m_1, \ldots, m_{|\max U|}.$ For each $m_i \in \max U,$ choose a prime ideal $P_i$ of $S = k[x,y]$ such that both $\height P_i = \height m_i$ and the fraction field $(S/P_i)_0$ is isomorphic to $k;$ note that if $\height m_i = 1$ for some $i,$ then $P_i = (x - i)$ works as a choice since $i \in k', S/{P_i} \cong k[y]$ and the fraction field of $k[y]$ is isomorphic to $k = k'(T).$ As the characteristic of $k'$ is zero, we do not have to worry about any two choices $(x - i)$ and $(x - j)$ coinciding unless $i = j.$  Since all the fraction fields of $S/P$ are $k,$ it follows that the residue fields $S_P/PS_P$ of the local rings $S_P$ are $k.$ 

Enumerate the chosen prime ideals as $P_1, \ldots, P_{|\max U|}.$ Adjoin an indeterminate $z$ to $S$ to form $S[z].$ Localize at the union of the prime ideals $\bigcup_{i=1}^n P_iS[z]$ to form $R,$ and let $\varphi: U \rightarrow \Spec R$ send $\min U$ to $(0),$ and $m_i$ to $P_iR.$ In order to prove the theorem in this case, we need to define $\varphi$ on nonempty classes of the form $[\mathcal M/u_1].$ Write $\mathcal M = \{m_{i_1}, \ldots, m_{i_a}\}.$ By Lemma \ref{genericufd}, there are $\beta$ prime ideals residing in $[\mathcal P/(0)]$ where $\mathcal P = \{P_{i_1}R, \ldots, P_{i_a}R\}.$ In particular, we may arbitrarily send $[\mathcal M/u_1]$ into $[\mathcal P/(0)]$ via a suitable injection. Since $\mathcal H_U = \{u_1\},$ the map trivially satisfies all the necessary properties stated in the theorem. 

 Let $M_{\mathcal H_U}$ be the set of all maximal nodes of $U$ such that $m \ge h$ for some $h \in H_1\cap \mathcal H_U,$ and for each $u_i \in \min U,$ let $M_{u_i}$ be the set of all $m \in \max U$ such that $L(m) \cap \min U = \{u_i\};$ in other words, $M_{u_i}$ consists of those maximal nodes of $U$ that dominate only one minimal node in $U,$ namely $u_i.$ We make an important observation about $M_{\mathcal H_U}.$ If $m \in M_{\mathcal H_U}$ and $h$ is an element in $H_1 \cap \mathcal H_U$ such that $m \ge h,$ then $h$ is unique. Otherwise $\Lambda_{L_U(m)}$ would contain at least two elements, contradicting the simplicity of $U.$ 

Write $a = |M_{\mathcal H_U}|, b_i = |M_{u_i}|$ for all $1 \le i \le n.$ Notice that $\max U$ is partitioned into $M_{\mathcal H_U}$ and $\bigcup_{i=1}^n M_{u_i},$ since if $m \in \max U$ and does not belong to any $M_{u_i},$ then it must dominate more than one minimal node. If $\height m = 1,$ then $m \in \mathcal H_U,$ a contradiction. Otherwise, if $\height m = 2,$ then since $U$ is simple we have that $L_U(m)$ is a $\beta$-tent and must therefore dominate some $h \in H_1\cap \mathcal H_{L_U(m)} \subseteq H_1 \cap \mathcal H_U.$ 

Now $\mathcal H_U$ is a finite, proper $K$-poset such that $h \in \max \mathcal H_U$ iff $\height h = 1$ in $\mathcal H_U.$ Enumerate the maximal nodes of $\mathcal H_U$ as $h_1, \ldots, h_c.$ Let $\mathcal H_U'$ be a simplification of $\mathcal H_U.$ Now every maximal node of $\mathcal H_U'$ has height one and since $d_{L_{\mathcal H_U'}(m)}(m) = 1$ for all $m \in \max \mathcal H_U',$ it follows that $\mathcal H_U'$ is a finite disjoint union of finite fans. We get $\mathcal H_U'$ as $\Spec R_0$ by choosing a suitable localization $R_0$ of a direct product of $n$ copies of $k[x]$ satisfying $R_0/q \cong k$ for all $q \in \max \Spec R_0.$ By Theorem \ref{iso}, there exists an isomorphism $\varphi_{\mathcal H_U}$ between $\mathcal H_U$ and the spectrum of some Noetherian subring $R_1 \subseteq R_0$ satisfying $R_1/q \cong k$ for all $q \in \max \Spec R_1.$ As an additional important note, part ``c" of \cite[Theorem A]{DL} implies that $(R_1/p)_0 \cong (k[x])_0 \cong k$ for all minimal prime ideals $p$ of $R_1$ by choice of our rather large field $k.$ 

Let $R_2 = R_1[x, y],$ and enumerate the elements of $M_{\mathcal H_U}$ as $m_1, \ldots, m_a.$ What we will do is choose prime ideals of $R_2$ that shall correspond to those in $M_{\mathcal H_U}$ under a map $\varphi$ we have yet to define. Indeed, each $m_i$ dominates some unique $h_j$ for $1 \le j \le c.$ Set $P_i = (\varphi_{\mathcal H_U}(h_j), x - i)R_2.$ Note that we have $(R_2/P_i)_0 \cong k.$ Now fix $1 \le i \le n,$ and consider the set $M_{u_i}.$ Enumerate the elements of it as $m_{u_i}^{(1)}, \ldots, m_{u_i}^{(b_i)}.$ Let $1 \le j \le b_i.$ If $\height m_{u_i}^{(j)} = 2,$ set $P_{u_i}^{(j)} = (\varphi_{\mathcal H_U}(u_i), x - j, y - j).$ If $\height m_{u_i}^{(j)} = 1,$ then set $P_{u_i}^{(j)} = (\varphi_{\mathcal H_U}(u_i), x - j).$ By the note at the end of the preceding paragraph, note that $R_1/(\varphi_{\mathcal H_U}(u_i)) \cong k$ so that, similar to previous reasoning, $(R_2/P_{u_i}^{(j)})_0 \cong k$ for all $1 \le j \le b_i.$ Adjoin an indeterminate $z$ to $R_2,$ and let $R$ be the localization of $R_2[z]$ at the union $\bigcup_{i=1}^a P_{i}R_2[z] \bigcup_{j_1=1}^{n} \bigcup_{j_2 = 1}^{b_{j_1}} P_{u_{j_1}}^{(j_2)}R_2[z].$ Note that all the residue fields of $R$ are $k$ by construction, and of course $|R| = |k|.$ Apply Lemma \ref{genericufd} to our chosen extended prime ideals and $R_2.$ This establishes property 1. 

Define a map from $U$ to $\Spec R$ as follows. If $u \in \max U,$ then simply define $\varphi(u)$ to be $PR$ where $P$ is the maximal prime ideal we picked to correspond to $u$ in the previous paragraph. $P$ was chosen to survive the localization, and thus $PR \in \max \Spec R.$ If $u \in \mathcal H_U,$ define $\varphi(u) = \varphi_{\mathcal H_U}(u)R.$ Since $\varphi_{\mathcal H_U}(u)$ is contained in a prime ideal $P$ surviving the localization, it follows that $\varphi_{\mathcal H_U}(u)R \in \Spec R.$ If $u \in U$ and resides neither in $\mathcal H_U$ nor $\max U$ then $\height u = 1$ and there exists a nonempty subset $\mathcal N \subseteq \max U$ for which $G(u)^* = \mathcal N.$ Moreover, since $u \notin \mathcal H_U,$ it follows that $L(u)^*$ cannot have more than one element in it and thus is equal to $\{u_i\}$ for some minimal node $u_i$ of $U.$ That is, $u \in [\mathcal N/ u_i].$ In particular, we need only find injective maps $\psi_{[\mathcal N/u_i]}$ from all nonempty sets of the form $[\mathcal N/u_i]$ into $[\varphi(\mathcal N)/\varphi(u_i)]$ in order to complete our definition of $\varphi.$  Fixing $[\mathcal N/u_i],$ observe  $|[\mathcal N/ u_i]| \le \beta$ and since $R$ satisfies the conditions of Lemma \ref{genericufd}, it follows that $[\varphi(\mathcal N)/\varphi(u_i)]$ has exactly $\beta$ prime ideals in it (recall that we have already defined $\varphi(\mathcal N) \subseteq \max \Spec R$ (resp. $\varphi(u_i)$)). Choose any injective set map $\psi_{[\mathcal N/u_i]}: [\mathcal N/u_i] \rightarrow [\varphi(\mathcal N)/\varphi(u_i)],$ and simply declare $\varphi(u) = \psi_{[\mathcal N/u_i]}(u)$ in the case where $u \in [\mathcal N/u_i].$ Checking that $\vp$ has all the required properties is straightforward. \end{proof}

\section{Proof of Main Result}\label{proof of main result}

Given a proper, simple $K$-poset $U,$ Theorem \ref{Embed} provides a commutative Noetherian ring $R$ for which $U$ embeds into $\Spec R$ in a very nice way. Of course, $\Spec R$ need not be isomorphic to $U,$ but the only reason that $U$ and $\Spec R$ are not isomorphic is because some class $[\mathcal M/u]$ is strictly smaller than $[\varphi(\mathcal M)/\varphi(u)];$ that is, if $U, \varphi$ and $\Spec R$ satisfy the conclusions of the lemma with the additional property that $|[\mathcal M/u]| = |[\varphi(\mathcal M)/\varphi(u)]|$ for all classes $[\mathcal M/u],$ then there must exist an isomorphism $\psi$ from $U$ onto $\Spec R.$ In order to realize $U$, then, our work is thus reduced to removing prime ideals from intersections in a certain sense. S. Wiegand found herself in a similar position in \cite{SPE}, and it is from there that we find much of our intuition throughout the rest of this paper.

\begin{lem}\label{intersect} Let $R$ be a reduced Noetherian ring with finitely many maximal prime ideals $M_1, \ldots, M_n.$ Suppose $T_1, \ldots, T_n$ are reduced faithfully flat extensions of $R$ with common total quotient ring $\mathcal Q$ for which $PT_i$ is a prime ideal for all prime ideals $P$ of $R$ and $1 \le i \le n.$ Set $R_i = (T_i)_{M_iT_i}^{\circ} \subset \mathcal Q,$ and define $S: = \bigcap_{i=1}^n R_i.$ If each $R_i$ is Noetherian and $T_i \subseteq S$, then the following statements hold:
\begin{enumerate}
\item[1] If $\mathcal Q_S$ is the total quotient ring of $S,$ then $\mathcal Q = \mathcal Q_S.$
\item[2] If $M$ is a maximal prime ideal of $S,$ then $S_{M}^{\circ} = R_j$ for some $j.$
\item[3] If $J$ is an ideal of $S$ and there is $i$ for which $JR_i \cap S \subseteq JR_j \cap S$ for all $ 1 \le j \le n,$ then $JR_i \cap S = J.$ 
\item[4] $S$ is Noetherian and the map $Q \rightarrow QS$ sends maximal prime ideals in $R$ onto maximal prime ideals in $S$ and minimal prime ideals in $R$ onto minimal prime ideals in $S.$ Moreover, $\height M_i = \height M_iS$ for all $1 \le i \le n.$ 
\end{enumerate}
\end{lem}

\begin{proof}

\textbf{Item 1:} If $t \in \mathcal Q_S,$ then $t = a/b$ for $a, b\in S$ and $b$ a non zero divisor $S.$ Since $S \subseteq R_i \subseteq \mathcal Q,$ we may write $a = a'/a''$ and $b = b'/b''$ for $a', a'', b', b'' \in T_i$ such that neither $b''$ nor  $a''$  are zero divisors of $T_i.$ In particular, $b'$ cannot be a zero divisor of $T_i$ since $b$ is not a zero divisor of $S$ and therefore $a/b = b''(a'/a'')/b' \in \mathcal Q.$ Conversely, if $a/b \in \mathcal Q$ for some $a, b \in T_i,$ with $b$ not a zero divisor of $T_i,$ then of course $a, b \in S.$ If $b$ is annihilated by a nonzero $c \in S \subseteq R_i,$ then $c = c'/c''$ for $c', c'' \in T_i$ with $c''$ not a zero divisor in $T_i$ so that $c'$ annihilates $b$ in $T_i,$ a contradiction. Therefore $a/b \in \mathcal Q_S$ and $\mathcal Q = \mathcal Q_S$ as desired.

\textbf{Item 2:} Let $N$ be a maximal prime ideal of $S,$ and let $x \in N.$ If $x$ is invertible in each $R_i,$ then $x$ must be invertible in $S = \bigcap_{i=1}^n R_i.$ Therefore, there exists $R_i$ for which $xR_i \subsetneq R_i$ for some $i.$ If $x \notin M_iR_i \cap S,$ then $x \in pR_i$ for some minimal prime ideal $p \subseteq M_l$ of $R.$ In particular, $x \in p\mathcal Q$ and since $\mathcal Q$ is the total quotient ring of $T_l$ and $R_l \subseteq \mathcal Q,$ we see that $x \in p\mathcal Q \cap R_l = pR_l \subseteq M_lR_l.$ As $x \in S,$ it immediately follows that $x \in M_lR_l \cap S.$ In particular, $N \subseteq \bigcup_{i=1}^n M_i R_i \cap S$ and seeing as $N$ is a maximal prime ideal of $S,$ we have by the prime-avoidance lemma that in fact $N = M_iR_i \cap S$ for some $i.$ To check that each $M_iR_i \cap S$ is a maximal ideal in $S,$ note that otherwise we would have $M_iR_i \cap S \subsetneq M_lR_l \cap S$ for some $l \ne i$ by what we have just shown and thus $M_iR_i \cap T_l \subset M_lT_l$ since $R_l$ is a localization of $T_l.$ Therefore, if $r \in M_i$ in $R,$ then $r \in M_iR_i \cap T_l \cap R \subset M_lT_l \cap R = M_l$ so that $M_i = M_l,$ an absurdity. Thus, each $M_iR_i \cap S$ for $1 \le i \le n$ is indeed a maximal prime ideal of $S.$

Fixing $1 \le j \le n,$ we claim $S_{M_jR_j \cap S}^{\circ} = R_j.$ If $x \in S_{M_jR_j \cap S}^{\circ}$ then $x = a/b$ is a ratio of elements in $S$ where $b$ is not in any of the minimal prime ideals of $S,$ and not an element of $M_jR_j \cap S.$ Since $S \subseteq R_j$ and $\mathcal Q = \mathcal Q_S,$ we see that $b$ cannot be in any of the minimal prime ideals of $R_j$ and certainly cannot reside in $M_jR_j.$ Therefore $a/b \in R_j$ as desired. Conversely, if $x \in R_j,$ then $x = c/d$ for $c, d \in T_j$ but not in $M_jT_j$ nor any of the minimal prime ideals of $T_j.$  Again invoking the fact that $\mathcal Q = \mathcal Q_S,$ we see that $d$ cannot lie in any of the minimal prime ideals of $S,$ and of course $d \notin M_jR_j \cap S$ so that $c/d \in S_{M_jR_j \cap S}^{\circ}.$ It follows that all the localizations at maximal prime ideals $N$ of $S$ coincide with some $(R_i)_{M_iR_i}.$

\textbf{Item 3:} To show that $(JR_i \cap S)S_N \subseteq JS_N$ for all maximal prime ideals $N$ of $S,$ we need only show that $(JR_i \cap S)(R_j)_{M_jR_j} \subseteq J(R_j)_{M_jR_j}$ all $1 \le j \le n.$ Fix $j.$ We have $(J(R_j)_{M_jR_j} \cap R_j)(R_j)_{M_jR_j} = J(R_j)_{M_jR_j}$ and since we are assuming $JR_i \cap S \subseteq JR_j \cap S,$ it follows that $(JR_i \cap S)(R_j)_{M_jR_j} \subseteq J(R_j)_{M_jR_j}$ as desired.

\textbf{Item 4:} Since $S$ is locally Noetherian and it has only finitely many maximal prime ideals, it follows that $S$ is Noetherian by a result of Nagata; see \cite[Exercise 9.6]{EISENBUD}. If $M_i \in \max \Spec R,$ then it immediately follows that $M_iS = M_iR_i \cap S$ upon applying item 3; indeed $M_iR_i \cap S \subseteq M_jR_j \cap S = S$ for all $j \ne i.$ Now assume $p \in \min \Spec R.$ Let $1 \le j \le n.$ Now $pR_j$ is a minimal prime ideal of $R_j$ and since $\mathcal Q$ is a localization of $R_j$ preserving $pR_j,$ we see that $p\mathcal Q \cap R_j = pR_j.$ Therefore, $pR_i \cap S \subseteq p\mathcal Q \cap S \subseteq (p\mathcal Q \cap R_j) \cap S = pR_j \cap S.$ Therefore $pR_i \cap S = pS$ by item 3 and we see that $pS$ is thus a prime ideal. It is a minimal prime ideal in $S$ since $R_i$ is a localization of $S$ and $pR_i$ is a minimal prime ideal of $R_i.$ Conversely, we see that all minimal prime ideals of $S$ must have the form $qS$ for some minimal prime ideal $q$ of $R$ since all of the minimal prime ideals in $S$ must be a contraction of a minimal prime ideal from $R_i.$ The final part of item 4 follows from basic localization properties, the altitude theorem, and the faithful flatness assumption. \end{proof}

We now make an observation about the cardinality of a Noetherian ring as it relates to the cardinalities of its residue fields. This lemma is a very mild refinement of \cite[Lemma 2.1]{KO}.

\begin{lem}\label{inequality} Let $R$ be any commutative semi-local Noetherian domain of dimension two, and let $M$ be any maximal prime ideal of $R.$ If $\height M = 2,$ then $$\aleph_0 + |R/M| \le |[M/0]| \le |\Spec R| \le |R| \le |R/M|^{\aleph_0}.$$ \end{lem}
\begin{proof} 

Enumerate the maximal prime ideals of $R$ as $M_1 := M, M_2, \ldots, M_n.$ Let $\{z_{\alpha}\}$ be a system of representatives for $R/M.$ Choose any nonzero $g \in \bigcap_{i=1}^n M_i,$ enumerate the minimal prime ideals over $g$ as $Q_1, \ldots, Q_m,$ and choose $f \in M\setminus \left( \bigcup_{i=2}^n M_i \bigcup \bigcup_{j=1}^m Q_j \right).$ Define $w_{\alpha} = f + z_{\alpha}g$ for all $\alpha.$ Note that $w_{\alpha} \notin M_i$ for any $1 < i \le n.$ Consequently, if $w_\alpha \in P,$ then the only maximal prime containing $P$ must be $M;$ that is, $P \in [M/0].$ Moreover, no two $w_{\alpha}$ may reside in a common height one prime ideal by our choices of $f$ and $g$. The altitude theorem guarantees a height one prime ideal $P_{\alpha}$ containing each $w_{\alpha},$ and consequently $|[M/0]| \ge |R/M|.$ Of course, there must be infinitely many elements in $[M/0]$ regardless of the cardinality of $R/M,$ so we get $\aleph_0 + |R/M| \le |[M/0]| \le |\Spec R|.$

For the third inequality, note that the set of all nonempty finite subsets of $R$ has the same cardinality as $R.$ Since $R$ is Noetherian, the map sending a nonempty finite set $F$ to the ideal $(F)$ in $R$ is surjective, and it follows that the set of all ideals in $R$ has cardinality not greater than $R.$ Therefore $|\Spec R| \le |R|.$ 

The last inequality follows from the statement of \cite[Lemma 2.1]{KO}. \end{proof}

\begin{lem}\label{union} Let $\Lambda$ be an infinite well ordered set of cardinality $\beta$ with the property that $|S(a)^*| < \beta$ for all $a \in \Lambda.$  If $X \subseteq \Lambda$ is any nonempty subset of $\Lambda$ of cardinality $\beta,$ then $X$ may be expressed as a disjoint union of $\beta$ sets $X_{a}$ indexed by $a \in \Lambda$ in such a way that $|X_a| = \beta$ and $\min X_a \ge a$ for all $a \in \Lambda.$ 
 \end{lem}
 
 \begin{proof} For simplicity, assume $\min \Lambda = 0.$ Let $\mathcal X$ be a collection of $\beta$ subsets of $X$ of cardinality $\beta$ such that the intersection of any two is empty and their union is all of $X.$ For each $w \in X,$ let $T_w$ be the unique set in $\mathcal X$ containing $w.$ Let $x$ be the least element of $X,$ and define $X_0 = T_x.$ Let $a \in \Lambda,$ and assume that we have defined $X_{a'}$ for all $0 \le a' < a$ in such a way that $\min X_{a'} \ge a'.$ Since $|S(a)^*| < \beta,$ we must have only defined fewer than $\beta$ sets. In particular, $|\mathcal X\setminus\{X_{a'}: 0 \le a' < a\}| = \beta.$ Let $Y$ be the set of all elements $y' \in X$ such that $y'$ is the least element of some set in $\mathcal X\setminus\{X_{a'}: 0 \le a' < a\}.$ Since the latter set has cardinality $\beta,$ it follows that $|Y| = \beta.$ In particular, if $y' < a$ for all $y' \in Y,$ then $|S(a)^*| = \beta,$ a contradiction to our assumption about $\Lambda.$ Therefore, some element $y \in Y$ must satisfy $y \ge a.$ Relabel $T_y$ as $X_a.$ 
 
 We have inductively created $\beta$ sets $X_a$ of cardinality $\beta$ such that $\min X_a \ge a$ for all $a \in \Lambda.$ If $X = \bigcup_{a \in \Lambda} X_a,$ then we are done. Otherwise, if $X \ne \bigcup_{a \in \Lambda} X_a,$ then the set $X'' = X \setminus \bigcup_{a \in \Lambda} X_a$ has least element $z.$ Declare $X_a' = X_a$ for all $a \ne z,$ and $X_z' = X_z \cup X''.$ Then still $\min X_a' \ge a$ for all $a \in \Lambda$ and of course $X = \bigcup_{a \in \Lambda} X_a'.$  \end{proof}
 
 \begin{thm}\label{class3} Let $\beta$ be an infinite cardinal, and let $k'$ be a countable field of characteristic zero. Set $k = k'(T),$ where $T$ is a set of $\beta$ indeterminates over $k'.$ Let $R$ be a reduced Noetherian $k$-algebra such that $|R| = \beta$ and $\Spec R$ is a $K$-poset with $\ell$ maximal elements and $|[\mathcal M/Q]| = \beta$ for all classes $[\mathcal M/Q] \subset \Spec R$ such that $Q \subset M$ for all $M \in \mathcal M.$ For each class $[\mathcal N/P] \subset \Spec R,$ let $\alpha_{[\mathcal N/P]} \le |[\mathcal N/P]|$ be any cardinal number. Then there exists a Noetherian $k$-algebra $S \supseteq R$ with $\ell$ maximal prime ideals for which the following items hold:
\begin{enumerate}
\item[1] If all residue fields of $R$ are $k,$ then all the residue fields of $S$ are also $k.$ Furthermore, $|S| = |k| = \beta.$
\item[2] We have $|[\mathcal MS/QS]| = \alpha_{[\mathcal M/Q]}$ for all intersection classes $[\mathcal M/Q] \subset \Spec R$ with $|\mathcal M| \ge 2$ and $Q \subset M$ for all $M \in \mathcal M.$   
\item[3] There is an isomorphism $\psi$ from $\mathcal H_{\Spec R} \cup \max \Spec R$ onto $\mathcal H_{\Spec S} \cup \max \Spec S$ given by $\psi(Q) = QS.$ 
\end{enumerate} 
\end{thm}

This theorem essentially has the effect of removing prime ideals from intersections while preserving mub information. It implies Theorem \ref{MT}. To see this, take any proper, simple $K$-poset $U$ of cardinality $\beta,$ and let $R_0, \vp$ be a Noetherian ring and poset map satisfying the conclusions of Theorem \ref{Embed}. Of course, $\vp$ will almost never be an isomorphism since there may exist classes $[\mathcal M/u]$ such that $|[\mathcal M/u]| < |[\vp(\mathcal M)/\vp(u)]|.$ However, we may now simply apply Theorem \ref{class3} to obtain $S \supseteq R_0$ that is Noetherian and has the right number of prime ideals in all the classes of the form $[\mathcal N/Q]$ with $|\mathcal N| \ge 2.$ The remaining nonempty classes in $\Spec S$ all have cardinality $\beta$ by Lemma \ref{inequality} and since $U$ is a proper $K$-poset, all its nonempty classes of the form $[m/u]$ for a maximal node $m$ of height two and a minimal node $u$ are of cardinality $\beta$ as well. In particular, an isomorphism is evident.

We now begin the proof of Theorem \ref{class3}. The proof is largely inspired by the proof of \cite[Proposition 1]{SPE}. Let $\Lambda$ be any set of cardinality $\beta,$ and place a well order on $\Lambda$ with least elements $0, 1$ so that if $t \in \Lambda,$ then $|S(t)^*| < \beta.$  Note that $\Lambda$ does not have a greatest element. Let $X$ be a set of indeterminates over $R$ indexed by $\Lambda.$ For each $x \in R[X],$ let $\mathcal F_x$ be the set of height one prime ideals minimal over $x.$ Note that $|\mathcal F_x|$ is finite for all $x \in R[X].$ 

Enumerate the classes $[\mathcal M/Q] \subset \Spec R$ with $|\mathcal M| \ge 2$ and $\alpha_{[\mathcal M/Q]} = \beta$ as $[\mathcal M_1/Q_1], \ldots, [\mathcal M_n/Q_n].$ Note that this is an enumeration of large classes, and some of the minimal prime ideals $Q_j$ of $R$ may appear more than once or not at all. Define $\mathcal L_i = \{QR[X]: Q \in [\mathcal M_i/Q_i]\}.$ Assume that we have not enumerated all the classes since otherwise we can simply take $S = R$ and be done. For each class $[\mathcal N/P]$ with $\alpha_{[\mathcal N/P]} < |[\mathcal N/P]|$ and $|\mathcal N| \ge 2,$ arbitrarily select $\alpha_{[\mathcal N/P]}$ prime ideals from $[\mathcal N/P] \subset \Spec R.$ Let $\mathcal S_{[\mathcal N/P]}$ be the extensions to $R[X]$ of the selected set of prime ideals along with the extensions to $R[X]$ of the finite set $\mathcal H_{\Spec R}.$ Collect all such prime ideals into a single set: $$\mathscr S:= \bigcup_{\substack{ [\mathcal N/P] \\ \alpha_{[\mathcal N/P]} < \beta}} \mathcal S_{[\mathcal N/P]}.$$ 

Let $\mathscr D$ be the set of height one prime ideals $P$ of $R[X]$ for which one the following conditions are met:

\begin{enumerate}
\item[1] $P \cap R$ is a minimal prime ideal of $R;$ or
\item[2] $(P \cap R)R[X]$ is different from every prime ideal in $\bigcup_{i=1}^n \mathcal L_i \bigcup \bigcup_{Q \in \mathscr S} Q.$
\end{enumerate}

Since $|R[X]| = \beta$ and $R$ is Noetherian, every height one prime ideal of $R[X]$ is finitely generated. In particular, $|\mathscr D| \le \beta.$ Since we are assuming we have not enumerated all of the classes, some $\alpha_{[\mathcal N/P]} < \beta = |[\mathcal N/P]|.$ Seeing as all the prime ideals in that class extend to $R[X],$ we must have $\beta \le |\mathscr D|$ as well so that in fact $|\mathscr D| = \beta.$  Enumerate the prime ideals in $\mathscr D$ as $\{P_t\}_{t \in \Lambda}.$ Note that any prime ideal $P_t$ in our enumeration has fewer than $\beta$ predecessors by choice of $\Lambda.$ 

We now begin the rather delicate process of choosing elements in $R[X]$ to form a set $\mathcal J^*$ that will be used to split prime ideals in $\mathscr D.$ For each $t \in \Lambda$ and $1 \le i \le n,$ we will define a set $\mathcal K_i^t \subset \mathcal L_i.$

We start with the base step $t = 0.$ Choose arbitrarily one element from each $\mathcal L_i$ to form a one element set $\mathcal K_i^0.$ Let $P_0 \in \mathscr D$ be the first prime ideal in our enumeration. Since $P_0 \notin \bigcup_{i=1}^n \mathcal K_i^0 \bigcup \mathscr S,$ and $\left | \bigcup_{i=1}^n \mathcal K_i^0 \bigcup \mathscr S \right | < \beta,$ Lemma \ref{primeavoidance} implies the existence of an element $$y_{P_0} \in P_0 \setminus \left ( \bigcup_{i=1}^{n} \bigcup_{Q \in \mathcal K_i^0} Q \bigcup \bigcup_{Q' \in \mathscr S} Q' \right ).$$ Define $\mathcal J_0 = \{1, y_{P_0}\};$ we include the element 1 in our set to ensure $\mathcal J_0 \cap R$ is nonempty, a property we will need later. This completes the first step.

Now assume $t \in \Lambda$ is positive and for each $s < t$ we have chosen sets $\mathcal K_i^s, \mathcal J_s$ and elements $y_{P_s} \in R[X]$ such that:
\begin{enumerate}
\item[1] Both $|\mathcal K_i^s| < \beta$ and $\mathcal K_i \subset \mathcal L_i.$ If $s$ is minimal with respect to $Q \in \mathcal K_i^s$ for some prime ideal $Q,$ then $Q \notin \bigcup_{s' < s} \mathcal F_{y_{P_{s'}}};$ 
\item[2] For all $s' \le s,$ we have $\mathcal K_i^{s'} \subseteq \mathcal K_i^s,$ and there exists a prime ideal $Q_i^s \in \mathcal K_i^s \setminus \left( \bigcup_{s' < s} \mathcal K_i^{s'}\right);$
\item[3] Each $y_{P_s}\in P_s \setminus \left (\bigcup_{i=1}^n \bigcup_{Q \in \mathcal K_i^s} Q \bigcup \bigcup_{Q' \in \mathscr S} Q' \right);$ and  
\item[4] Each set $\mathcal J_s$ contains $y_{P_s}$ and $|\mathcal J_s| < \beta;$ the inclusion $\mathcal J_{s'} \subseteq \mathcal J_{s}$ is true for all $s' \le s;$ and whenever $Q \in \mathscr S \bigcup \bigcup_{i=1}^n \mathcal K_i^s,$ then $Q \cap \mathcal J_s$ is empty. 
\end{enumerate} 

We need to show that we can make all those choices again for $t.$ Since $\left |\bigcup_{s < t} \mathcal F_{y_{P_s}}\bigcup \bigcup_{s < t} \mathcal K_i^{s} \right| < \beta,$ the set $$\mathcal L_i \setminus \left(\bigcup_{s < t} \mathcal F_{y_{P_s}}\bigcup \bigcup_{s < t} \mathcal K_i^{s} \right)$$ \noindent has cardinality $\beta.$ Choose arbitrarily any element $Q_i^t$ in the above set, and define $\mathcal K_i^t = \{Q_i^t \} \bigcup_{s < t} \mathcal K_i^s.$ Since $P_t \in \mathscr D,$ and $|S(t^*)| < \beta,$ we may use Lemma \ref{primeavoidance} to choose an element $$y_{P_t} \in P_t \setminus \left (\bigcup_{i=1}^n \bigcup_{Q \in \mathcal K_i^t} Q \bigcup \bigcup_{Q' \in \mathscr S} Q' \right ).$$  Define $\mathcal J_t = \{y_{P_t}\} \bigcup \bigcup_{s < t} \mathcal J_s.$ It is easily shown that our choices of $Q_i^t, \mathcal K_i^t, y_{P_t}$ and $\mathcal J_t$ satisfy conditions 1-4. Set $\mathcal J^* = \bigcup_{t \in \Lambda} \mathcal J_t$ and $\mathcal K_i = \bigcup_{t \in \Lambda} \mathcal K_i^t.$

We are now ready to begin forming our desired ring $S.$ First note that $X$ trivially inherits a partial ordering given by $x_t \le x_s$ iff $t \le s \in \Lambda.$ Define $R_0 = R, R_1 = R[x_0, x_1].$ Let $t > 1,$ and having defined $R_s$ for all $s < t,$ define $R_t = (\bigcup_{s < t} R_s)[x_t].$ Of course, $R[X] = \bigcup_{t \in \Lambda} R_t.$ Express $X$ as a disjoint union of sets $X_1, \ldots, X_{\ell}$ such that $X_i$ has cardinality $\beta$ for all $1 \le i \le \ell.$ Fix $i,$ and use Lemma \ref{union} to write each $X_i$ as a disjoint union indexed as $X_i^a$ for $a \ge 0$ in such a way that $\min X_i^a \ge a$ for all $a \in \Lambda.$ For each $i,$ we will carefully choose a surjective set map from $\sigma_i: X_i \rightarrow \mathcal J^*$ that will have the property that whenever $x_t \in X_i$ for some $t \in \Lambda\setminus\{0\},$ then $\sigma_i(x_t) \in R_{t'}$ for some $t' < t.$ Choose any surjection $\sigma_i^0$ from $X_i^0$ onto $\mathcal J^* \cap R.$ If $c$ is a limit element in $\Lambda,$ then let $\sigma_i^c$ be any surjection from $X_i^c$ onto $\mathcal J^* \cap R.$ Now assume $c$ is a successor element with immediate predecessor $c',$ and let $\sigma_i^{c}: X_i^c \rightarrow \mathcal J^* \cap R_{c'}$ be any surjective map. Having defined $\sigma_i^c: X_i^c \to \mathcal J^* \cap R_{c'}$ for all $c \in \Lambda,$ there is a naturally defined set map $\sigma_i$ from $X_i$ onto $\mathcal J^*.$ If $x_t \in X_i,$ then $x_t \in X_i^u$ for some $u.$ In particular, $t \ge u$ by Lemma \ref{union} and $\sigma_i(x_t)$ is either in $R$ or $R_{u'}$ where $u'$ is an immediate predecessor to $u.$ If $\sigma_i(x_t) \in R:=R_0,$ then of course $\sigma_i(x_t) \in R_{t'}$ for some $t' < t$ since $t > 0.$ Otherwise, $\sigma_i(x_t) \in R_{u'}$ with $u' < u \le t,$ and the desired placement of $x_t$ still holds. 

Let $Y_j := \{\sigma_j(z)/z: z \in X_j\}, Y := \bigcup_{j=1}^{\ell} Y_j.$ Set $T_j := R[X\setminus X_j, Y_j]$ and $S_j = (T_j)_{M_jT_j}^{\circ}.$ Define $S := \bigcap_{j=1}^{\ell} S_j.$ Note that each $S_j$ is Noetherian by \cite[Lemma 7]{JSPEC}. We now show that $S$ is our desired ring.

\begin{lem}\label{property1} $S \supseteq R[X,Y]$ is a Noetherian extension of $R[X,Y]$ satisfying the conclusions of Lemma \ref{intersect}. Moreover, $|S| = \beta$ and all the residue fields are $k.$ Consequently, $|\Spec S| \le \beta.$ \end{lem}

\begin{proof} We argue by induction that $R[X] \subseteq S,$ and then we use that result to show that $R[X,Y] \subseteq S.$ Clearly $R_0 \subset S.$ Let $t > 0,$ and assume that we have shown $R_s \subset S$ for all $0\le s < t.$ Let $1 \le j \le \ell$ arbitrary. We need to show that $R_t \subseteq S_j.$ If $x_t \in X \setminus X_j,$ then $x_t \in T_j \subseteq S_j.$ Otherwise, $\sigma_j(x_t)/x_t \in S_j.$ Clearly, $\sigma_j(x_t)/x_t \notin M_jT_j$ and is not a zero divisor in $S_j$ because $\mathcal J^*$ contains no zero divisors. In particular, $\sigma_j(x_t)/x_t$ is a unit in $S_j.$ It follows that $x_t/\sigma_j(x_t) \in S_j$ and since $\sigma_j(x_t) \in S_j$ by induction, we see that $x_t \in S_j$ as well. It follows that $R[X] \subseteq S.$ 

We argue similarly to show that $Y \subseteq S.$ As before, fix $1 \le j \le \ell,$ and take $\sigma_l(z)/z$ for some $1 \le l \le \ell.$ If $j = l,$ then $\sigma_j(z)/z \in T_j \subseteq S_j.$ Otherwise, if $j \ne l$ then $z \in X_j,$ and since $X_j$ and $X_l$ are disjoint, we have $z \in X \setminus X_j.$ In particular, $z \in T_j$ and clearly $z \notin M_jT_j$ and is not a zero divisor in $S_j.$ It follows that $z$ is a unit in $S_j.$ Since $\sigma_j(z) \in R[X] \subset S_j,$ we see that $\sigma_j(z)/z \in S_j$ as desired. 

Each $T_j$ is of course a faithfully flat extension of $R,$ and seeing as $T_j \subset S$ for all $j,$ they all clearly share a common total quotient ring $\mathcal Q$ while also satisfying the conditions of Lemma \ref{intersect}. $S$ thus satisfies the conclusions of the lemma as desired, and $S$ is thus seen to be a Noetherian extension of $R[X,Y].$ As a further consequence of Lemma \ref{intersect}, computing the residue fields of $S_j$ comes down to understanding the residue fields of the maximal prime ideals each $S_j.$ In particular, if $R/M_j \cong k$ for all $M_j \in \max \Spec R,$ it follows that $S_j/M_jS_j$ is the rational function field over $k$ in $\beta$ indeterminates. As $S/M_jS \cong S_j/M_jS_j,$ it follows that any two residue fields of $S$ must coincide with $k.$

Finally, $|S| = \beta,$ and since $S$ is Noetherian it follows that $|\Spec S| \le |S| \le \beta.$ \end{proof}

\begin{lem}\label{ReduceToRX} Let $Q$ be any height one prime ideal of $R.$ If $QS_i \cap R[X] = QR[X]$ for all $i$ such that $Q \subset M_i,$ then $QS$ is a prime ideal of $S.$ \end{lem} 
\begin{proof} 

By part 3 of Lemma \ref{intersect}, it suffices to show $QS_i \cap S \subseteq QS_j \cap S$ for all $i, j$ such that $Q \subset M_i \cap M_j.$ If $x \in QS_i \cap S,$ then $x \in S \subset S_j$ and there exists a non zero divisor $b \in T_j \setminus M_jT_j$ such that $bx \in QS_i \cap T_j.$ Choose a product $\xi$ of indeterminates from $X_j \cup X_i$ such that $\xi b x \in QS_i \cap R[X] = QR[X] = QS_j \cap R[X] \subset QS_j.$ Now $QS_j$ is a prime ideal of $S_j$ since $Q \subset M_j$ By Lemma \ref{property1}, $\xi, b, x \in S_j$ so that $\xi$ or $b \in QS_j$ if $x$ is not. If $\xi \in QS_j,$ then $x_t \in QS_j \cap R[X] = QR[X]$ for some $x_t \in X,$ an absurdity. Of course, $b \notin QS_j$ since $b$ is a unit in $S_j,$ and therefore  $x \in QS_j \cap S$ as desired. \end{proof}

For each $1 \le i \le \ell,$ let $S_{M_i}$ be the localization of $S_i$ at the maximal prime ideal $M_iS_i,$ and let $T_{M_i}$ be the localization of $T_i$ at $M_iT_i.$ Then $S_{M_i} = T_{M_i}.$ If $Q \subset M_i$ as in the previous lemma, then to prove that $QS_i \cap R[X] = QR[X],$ it suffices to show that the contraction of $QS_{M_i}$ to $R[X]$ under the canonical map is $QR[X].$

\begin{notation} For each $\omega \in \Lambda,$ define $X^{\omega} = \{x_t: t \le \omega\},$ and $Y_j^{\omega} = \{\sigma_j(z)/z: z \in X^{\omega} \cap X_j\}.$ \end{notation}

\begin{lem}\label{BetterExpression} Let $\omega \in \Lambda,$ and $Q$ a prime ideal of $R$ contained in $M_j.$ Suppose $h \in QS_{M_j} \cap R_{\omega}$ is nonzero in $T_{M_j}.$ Then $h = p/q$ for $p \in QR[ (X \setminus X_j) \cap X^{\omega}, Y_j^{\omega}],$ and an element $q \in R[ (X \setminus X_j) \cap X^{\omega}, Y_j^{\omega}]$ localizing to a unit in $S_{M_j}.$

\end{lem}

\begin{proof}
Assume $h \ne 0$ in $T_{M_j}$ since otherwise the result is trivial. Seeing as $S_{M_j} = T_{M_j},$ we may express $h \in QT_{M_j} \cap R_{\omega}$ as $f/g$ for nonzero $f \in QT_j, g\in T_j,$ and $g \notin M_jT_j.$  Since $QT_j$ is extended, all the $R$-coefficients forming $f$ are in $Q.$ 

It suffices to prove that if $z = x_{\eta}$ is the largest indeterminate of $X$ such that $z$ or $\sigma_j(z)/z$ appears in either polynomial expression of $f$ or $g$ and $\eta > \omega,$ then there exists $\eta' < \eta,$ and a new pair of elements $g', f' \in R[ (X \setminus X_j) \cap X^{\eta'}, Y_j^{\eta'}]$ such that $f' \in QR[ (X \setminus X_j) \cap X^{\eta'}, Y_j^{\eta'}], g' \in QR[ (X \setminus X_j) \cap X^{\eta'}, Y_j^{\eta'}], g'$ localizes to a unit and $h = f'/g'.$

Assume $\sigma_j(z)/z$ appears in either $f$ or $g.$ Then there exists $\delta < \eta$ such that both $$f = f_0(\sigma_j(z)/z)^a + \ldots + f_a,$$ \noindent and $$g = g_0(\sigma_j(z)/z)^b + \ldots + g_b$$ for $f_i, g_{i'} \in R[ (X \setminus X_j) \cap X^{\delta}, Y_j^{\delta}].$ Moreover, $f_i \in QR[ (X \setminus X_j) \cap X^{\delta}, Y_j^{\delta}]$ for all $0 \le i \le a.$ Let $c = \max\{a, b\},$ and choose a product $\xi$ of indeterminates from $X^{\delta}$ such that $\xi f_i, \xi g_{i'} \in R_{\delta};$ recall that if $\sigma_j(z)/z \in Y_j^{\delta},$ then $\sigma_j(z) \in R_{\delta}.$ Since $g$ is a unit in $T_{M_j},$ there exists $0 \le w \le b$ such that $g_w \notin M_jT_j.$ Compiling all this information, we may choose $r \in R \setminus M_j$ such that $r\xi z^c g h = r \xi z^c f.$ Rewriting the equation, we see that  $$r \xi h g_bz^c + \ldots + r \xi h g_0 \sigma_j(z)^bz^{c - b}$$ \noindent agrees with $$r \xi  f_az^c + \ldots + r \xi f_0 \sigma_j(z)^a z^{c - a}.$$ Since $h \ne 0$ as an element of $T_{M_j},$ and $g_w$ is a unit in $T_{M_j},$ we have $r \xi h g_w \sigma_j(z)^{b - w} \ne 0$ (recall each $\sigma_j(z)$ is regular). In particular, the coefficient of $z^{c - b + w}$ is nonzero, and algebraic independence dictates that it must coincide with the coefficient of $z^{c - a + w'}$ for some $0 \le w' \le a$ satisfying $c - b + w = c - a + w',$ or equivalently $b - w = a - w'.$ The coefficient of $z^{c - a + w'}$ is $r \xi \sigma_j(z)^{a - w'} f_{w'},$ and thus $r \xi h g_w \sigma_j(z)^{b-w} = r \xi \sigma_j(z)^{a - w'} f_{w'}.$ Solving for $h$ in $T_{M_j},$ and recalling that $b - w = a - w',$ we obtain $h = f_{w'}/g_w,$ which has the required form. Set $f' = f_{w'}, g' = g_{w},$ and $\eta' = \delta$ in this case. 

Now assume $z$ appears in either $f$ or $g.$ Express, in a similar fashion to the preceding work, $$f = f_0'z^d + \ldots + f_d'$$ \noindent and $$g = g_0'z^e + \ldots + g_e'$$ \noindent for $f_i', g_i \in R[ (X \setminus X_j) \cap X^{\varepsilon}, Y_j^{\varepsilon}]$ like before. Since $g$ is a unit in $T_{M_j}$, so is $g_t'$ for some $0 \le t \le e.$ Choose a product $\xi'$ of indeterminates from $X^{\varepsilon}$ such that $\xi f_i', \xi g_i' \in R_{\varepsilon}.$ Then $r' \xi' h g = r' \xi' f$ for $r' \in R \setminus M_j.$ Now $r' \xi' h g_t' \ne 0,$ and therefore $r \xi' h g_t'$ must coincide with $r' \xi' f_{t'}'$ for some $t'$ in the corresponding expression for $f.$ Solving for $h$ as before, we obtain $h = f_{t'}'/g_t'.$ Set $f' = f_{t'}', g' = g_t',$ and $\eta' = \varepsilon$ in this case. \end{proof}

The next theorem, which is really the heart of our construction, confirms that $\Spec S$ does indeed have the right prime ideal structure in its intersection classes.

\begin{thm}\label{HopeThisWorks} Let $Q$ be any prime ideal of $R$ for which $QR[X] \cap \mathcal J^*$ is empty. Then $QS$ is a prime ideal in $S.$ In particular, if $QR[X] \in \mathscr S \bigcup \bigcup_{i=1}^n \mathcal K_i$ then $QS$ is a prime ideal of $S.$ If additionally $Q \in [\mathcal W/Q'] \subset \Spec R$ is any prime ideal such that $QR[X] \cap \mathcal J^*$ is empty, then $Q \in [\mathcal WS/Q'S].$ Moreover, if $P$ is any height one prime ideal of $S$ contained in at least two maximal ideals, then $P$ is extended from $R.$ \end{thm}

\begin{proof} Let $Q \subset M_i$ be a prime ideal satisfying the conditions. We show inductively that $QS_{M_i} \cap R[X] = QR[X].$ Seeing as $M_i$ was chosen arbitrarily among those maximal prime ideals containing $Q,$ Lemma \ref{ReduceToRX} will show that $QS$ is a prime ideal in $S.$ Since $Q$ survives the localization to $S_{M_i},$ clearly $QS_{M_i} \cap R = Q,$ thus demonstrating the base case. 

Let $t > 0,$ and assume we have shown that $QS_{M_i} \cap R_s = QR_s$ for all $s < t.$  Take $x \in QS_{M_i} \cap R_t.$ If $x = 0$ in $S_{M_i} = T_{M_i},$ then $rx = 0$ for some $r \in R\setminus M_i.$ In particular, $rx \in QR_t,$ and since $r \notin Q,$ we must have $x \in QR_t.$ We may thus assume $x \ne 0$ in $T_{M_i}.$ Then $x = p/q$ for $p \in QR_t[Y_i^{t}], q \in R_t, q \in S_{M_i}^{\times}$ by Lemma \ref{BetterExpression}. In particular, there is a product $\xi$ of indeterminates chosen from $X_i^{t}$ for which $\xi q \in R[X]$ and $\xi q x \in QR_t.$ Assume, by way of contradiction, that $x \notin QR_t.$ Then $\xi q \in QR_t$ and thus $\xi \in QS_{M_i}.$ In particular, $z \in QS_{M_i}$ for some $z \in X_i,$ and it follows that $\sigma_i(z) = (\sigma_i(z)/z)z \in QS_{M_i} \cap R_{s} = QR_s$ for some $s < t,$ thus conflicting with our assumption that $QR_s \cap \mathcal J^*$ is empty. Therefore $x \in QR_t,$ and the induction is complete.  As $M_i \supset Q$ was chosen arbitrarily, we must have that $QS$ is a prime ideal in $S,$ by Lemma \ref{ReduceToRX}.

Assume additionally that $Q \in [\mathcal W/Q'].$ Then $QS \subset WS$ for all $W \in \mathcal W,$ and $QS \supset Q'S.$ If $MS \supset QS$ for some $M \notin \mathcal W,$ then since $QS \cap R = Q,$ we get that $M = MS \cap R \supset Q,$ which conflicts with $G(Q)^* = \mathcal W.$ In particular, $G(QS)^* = \mathcal WS.$ Similarly, if $Q''S \subset QS$ for some height zero prime ideal $Q''S \ne Q'S,$ then $Q'' = Q''S \cap R \subset QS\cap R = Q,$ which contradicts $L(Q)^* = Q'.$ 

Let $P \subset M_iS \cap M_jS$ be any prime ideal of height one in $S$ contained in two distinct maximal prime ideals. Choose a finite subset $F \subset X$ for which $\height P \cap R[F] \ge 1,$ and let $P^* \subset P \cap R[F]$ be any prime ideal of height one. If $P \cap R$ is a minimal prime ideal of $S,$ then $P^* \cap R$ is a minimal prime ideal of $S$ and thus $P^*R[X] \in \mathscr D$ since of course $\height P^*R[X] = 1.$ Then $P^*R[X] = P_s$ for some $s \in \Lambda,$ and therefore $y_{P_s} \in P.$ In particular, $\sigma_i(z) = (\sigma_i(z)/z)z \in P$ for some $z \in X_i.$ By Lemma \ref{property1}, both $\sigma_i(z)/z$ and $z$ are in $S$ so that either $z \in P$ or $\sigma_i(z)/z \in P.$ If $\sigma_i(z)/z \in P,$ then $\sigma_i(z)/z \in M_iS_i \cap T_i = M_iT_i,$ which is nonsense because $\sigma_i(z)/z$ is an indeterminate over $R$ in $T_i.$ Therefore, $z \in P.$ Now $z \in M_jS_j \cap T_j = M_jT_j$ since $z \in X \setminus X_j.$ Again, this is absurd since $z$ is an indeterminate over $R$ in $T_j.$ Therefore, $P \cap R$ is not a minimal prime ideal of $R.$ Seeing as $P \cap R$ is not a maximal prime ideal of $R,$ it follows that $\height P \cap R = 1.$ Either $(P \cap R)R[X] \cap \mathcal J^*$ is empty or not. If it is empty, then we have just shown that $(P \cap R)S$ is a prime ideal. Seeing as $\height P = 1,$ it follows that $P = (P \cap R)S.$ Otherwise, $(P \cap R)R[X] \cap \mathcal J^*$ is not empty, and thus again $\sigma_i(z) \in P$ for some $z \in X_i,$ conflicting with the assumption that $P$ is a prime ideal of $S.$ \end{proof}

Establishing the rest of property 2 amounts to showing that $[NS/PS]$ has cardinality $\beta$ for all maximal prime ideals $N$ of $S$ of height two such that $N \supset P.$ Since $k \subset S$ and $|S| = |k|,$ we may apply Lemma \ref{inequality} to $S$ to see that $\beta \le |[NS/PS]| \le |S| = \beta$ as desired.  

We now establish property 3 and hence the theorem. Since $M$ extends to $MS$ for all maximal prime ideals $M$ of $R,$ we need only show:

\begin{lem} $\mathcal H_{\Spec R} \cong \mathcal H_{\Spec S}$ via the map $Q \rightarrow QS.$ \end{lem}
\begin{proof} 
Property 4 of Theorem \ref{intersect} implies $Q \rightarrow QS$ sends $\min \Spec R$ onto $\min \Spec S.$ Moreover, since $\mathcal H_{\Spec R} \subset \mathscr S,$ it follows that $QS$ is a prime ideal for all $Q \in \mathcal H_{\Spec R}$ by Theorem \ref{HopeThisWorks}. In particular, if $Q$ dominates two minimal prime ideals of $R, QS$ must do the same in $S.$ Conversely, if $P$ is a height one prime ideal in $S,$ and it dominates two minimal prime ideals of $S,$ then $P \cap R$ must dominate two minimal prime ideals of $R.$ Therefore $P \cap R \in \mathscr S,$ and therefore $P = (P \cap R)S$ is extended by Theorem \ref{HopeThisWorks}. \end{proof}

\section*{References}
\bibliographystyle{plain}  
\bibliography{DissertationBib} 
\index{Bibliography@\emph{Bibliography}}

\end{document}